\documentclass[10pt, a4paper]{article}
\usepackage{amsfonts}
\usepackage{amsbsy}
\usepackage{amssymb}
\usepackage{amsmath}
\usepackage{amsmath,amscd}
\usepackage{amsthm}
\usepackage{fancyhdr}

\theoremstyle{plain}
\newtheorem{theorem}{Theorem}[section]

\newtheorem{corollary}{Corollary}[section]
\newtheorem{proposition}{Proposition}[section]
\theoremstyle{definition}
\newtheorem{definition}{Definition}[section]
\newtheorem{remark}{Remark}[section]

\begin{document}

\title{A note on para-holomorphic Riemannian Einstein manifolds}
\author{Cristian Ida, Alexandru Ionescu and Adelina Manea}
\date{}
\maketitle
\begin{abstract}
The aim of this note is the study of Einstein condition for para-holomorphic Riemannian metrics in the para-complex geometry framework. Firstly, we make some general considerations about para-complex Riemannian manifolds (not necessarily para-holomorphic). Next, using an one-to-one correspondence between para-holomorphic Riemannian metrics and para-K\"{a}hler Norden metrics, we study the Einstein condition for a para-holomorphic Riemannian metric and the associated  real para-K\"{a}hler Norden metric on a para-complex manifold. Finally, it is shown that every semi-simple para-complex Lie group inherits a natural para-K\"{a}hlerian Norden Einstein metric. 
\end{abstract}

\medskip 
\begin{flushleft}
\strut \textbf{2010 Mathematics Subject Classification:} 53C15, 53C25, 53C56. 

\end{flushleft}

\section{Introduction}
\setcounter{equation}{0}
A (holomorphic) complex Riemannian manifold is a complex manifold $M$, together with a (holomorphic) complex tensor field $G$ that is a complex scalar product (i.e., nondegenerate, symmetric, $\mathbb{C}$-bilinear form) on each
holomorphic tangent space of $M$. Geometrical aspects of the complex Riemannian manifolds with analytic (holomorphic) metrics and their applications to mathematical physics have been investigated by many authors, see for instance \cite{Du-Z, Ga-Bo, Ga-Iv1, Ga-Iv2, Iv, LeB, Ma, Pe, Wo}. The holomorphic Riemannian geometry possesses an underlying real geometry consisting of a pseudo-Riemannian metric of neutral signature for which the (integrable) almost complex structure tensor is anti-orthogonal. This leads to the notion of an anti-K\"{a}hlerian manifold (also known as K\"{a}hler-Norden manifold \cite{I-Sa, O1, O2, S} or $B$-manifold \cite{Ga-Bo}) that is a complex manifold with an anti-Hermitian metric and a parallel almost complex structure. In \cite{B-F-V}, it is proved that a metric on such a manifold must be the real part of a holomorphic metric. There is studied the Einstein condition for anti-K\"{a}hlerian metrics and it is shown that the complexification of a given Einstein metric leads to a method of generating new solutions of Einstein equations from a given one. Some generalized Einstein conditions on holomorphic Riemannian manifolds are studied in \cite{O2}.  

Although the almost product Einstein manifolds are studied in \cite{B-F-F-V}, our aim in this note is to formulate some analogue results as in \cite{B-F-V, O2} concerning the Einstein condition for para-holomorphic Riemannian metrics in terms of para-complex geometry.

The notion of almost para-complex structure (or almost product structure) on a smooth manifold was introduced in \cite{Lib1} and a survey of further results on para-complex geometry (including para-Hermitian and para-K\"{a}hler geometry) can be found for instance in \cite{Cru1, Cru2}. Also, other further signifiant developments are due in some recent surveys \cite{A-M-T, Cor1, Cor2}, where some aspects concerning the geometry of para-complex manifolds are presented sistematically by analogy with the geometry of complex manifolds using some para-holomorphic coordinate systems. See also \cite{Er2, Ka, Kr, La-S}.

The paper is organized as follows. In the Section 2, following \cite{A-M-T, Cor1, Ka}, we brief recall some basic notions used in the para-complex geometry. In Section 3, we define para-complex Riemannian metrics on para-complex manifolds, we prove that the real part of such a metric is a para-Norden (or almost product Riemannian) metric and following the construction from the complex case \cite{Ga-Iv1, Ga-Iv2}, we make some general considerations about the Levi-Civita and characteristic connections on para-complex Riemannian manifolds. Also, a  Schur type theorem concerning the para-holomorphic curvature is presented and we write Einsten equations in our setting.   In Section 4, starting from an one-to-one coorespondence between para-holomorphic Riemannian metrics and real para-K\"{a}hler Norden metrics on a para-complex manifold, we prove an  equivalence between Einstein condition with real constant for para-holomorphic Riemannian metrics and for the associated para-K\"{a}hler metric, giving an analogue result from the case of anti-K\"{a}hlerian Einstein manifolds \cite{B-F-F-V, B-F-V}. Also, the case when the para-holomorphic Riemannian metric is Einstein with a para-complex constant is also analysed in the similar manner with the case of K\"{a}hler-Norden metrics \cite{O2}. In the last Section, as an example of our study, it is shown that every semi-simple para-complex Lie group inherits a natural para-K\"{a}hlerian Norden Einstein metric. 

We notice that other problems related to generalized Einstein condition as in \cite{O2} can be adressed in the context of para-holomorphic Riemannian manifolds. Also, an important example of anti-K\"{a}hlerian Einstein metric on the tangent bundle of a space form is given in \cite{O-P} and a similar study can be also analysed in the context of para-complex geometry.  

The main methods used here are similarly and closely related to those used in the study  of complex Riemannian manifolds \cite{Ga-Iv1, Ga-Iv2, Iv} and anti-K\"{a}hlerian manifolds \cite{B-F-F-V, B-F-V, O2, S}.

\section{Preliminaries and settings in para-complex geometry}
\setcounter{equation}{0}
The algebra of para-complex numbers is defined as the vector space $C=\mathbb{R}\times\mathbb{R}$ with mulplication given by
\begin{displaymath}
(x_1,y_1)\cdot(x_2,y_2)=(x_1x_2+y_1y_2,x_1y_2+y_1x_2)\,,\,\forall\,(x_1,y_1), (x_2,y_2)\in C.
\end{displaymath}
Setting $e=(0,1)$, then $e^2=(1,0)=1$ and we can write $C=\mathbb{R}+e\mathbb{R}=\{z=x+ey\,|\,x,y\in\mathbb{R}\}$.

The conjugation of an element $z=x+ey\in C$ is defined as usual by $\overline{z}=x-ey$ and $Re \,z=x$ and $Im\,z=y$ are called the \textit{real part} and \textit{imaginary part} of the para-complex number $z$. 

A \textit{para-complex} structure on a real finite dimensional vector space $V$ is defined as an endomorphism $I\in End(V)$ which satisfy $I^2=Id$, $I\neq\pm Id$ and the following two eigenspaces $V^{\pm}:=\ker(Id\pm I)$ corresponding to the eigenvalues $\pm1$ of $I$ have the same dimension. Such a pair $(V,I)$ it is called a \textit{para-complex vector space}. Consequently, an \textit{almost para-complex structure} on a smooth manifold $M$ is defined as an endomorphism $I\in End(TM)$ with the property that $(T_xM,I_x)$ is a para-complex vector space, for every $x\in M$. Moreover, an almost para-complex structure $I$ on $M$ is said to be \textit{integrable} if the distributions $T^\pm M=\ker(Id\mp I)$ are both integrable, and in this case $I$ is called a \textit{para-complex structure} on $M$. A manifold $M$ endowed with a para-complex struture is called a \textit{para-complex manifold}. The para-complex dimension of a para-complex manifold $M$ is the integer $n = \dim_CM := (\dim_\mathbb{R}M)/2$.

Given two almost para-complex manifolds $(M,I_M)$ and $(N,I_N)$, a smooth map $f:(M,I_M)\rightarrow(N,I_N)$ is called \textit{para-holomorphic (respectively: anti-para-holomorphic)} if
\begin{equation}
df\circ I_M=I_N\circ df\,\,({\rm respectively:} \,\,df\circ I_M=-I_N\circ df).
\label{I1}
\end{equation} 
Moreover, an (anti-)para-holomorphic map $f:(M,I_M)\rightarrow C$ is called \textit{(anti)-para-holomorphic function}.

As usual, the Nijenhuis tensor $N_I$ associated to an almost para-complex structure $I$ is defined by 
\begin{equation}
\label{I2}
N_I(X,Y):=[IX,IY]-I[IX,Y]-I[X,IY]+[X,Y],
\end{equation}
for every $X,Y\in\Gamma(TM)$, and according to \cite{Cor1} we have that $I$ is integrable iff $N_I=0$. The Frobenius theorem implies, see \cite{Cor2}, the existence of local coordinates $(z_+^a,z_-^a)$, $a=1,\ldots,n=\dim_CM$ on para-complex manifold $(M,I)$, such that $T^{+}M=span\{\partial/\partial z_{+}^a\}$, $T^-M=span\{\partial/\partial z_{-}^a\}$, $a\in\{1,\ldots,n\}$. Such (real) coordinates are called \textit{adapted coordinates} for the para-complex structure $I$. 

As in the complex case, on every para-complex manifold $(M,I_M)$ we can define an atlas of para-holomorphic local charts $(U_\alpha,\varphi_\alpha)$, such that the transition functions $\varphi_\beta\circ\varphi_\alpha^{-1}:\varphi_\alpha(U_\alpha\cap U_\beta)\subset C^n\rightarrow\varphi_\beta(U_\alpha\cap U_\beta)\subset C^n$ are para-holomorphic functions in the sense of \eqref{I1}. Moreover, with every (real) adapted coordinate system $(z_+^a,z_-^a)$, $a\in\{1,\ldots,n\}$ on $U_\alpha$ we can associate a para-holomorphic system $(z^a)$, $a=1,\ldots,n$ by setting
\begin{equation}
\label{I4}
z^a=\frac{z_+^a+z_-^a}{2}+e\frac{z_+^a-z_-^a}{2}:=x^a+ey^a\,,\,\,a\in\{1,\ldots,n\}.
\end{equation}
According to \cite{Cor1}, $z^a$ are para-holomorphic functions in the sense of \eqref{I1} and the transition functions between two para-holomorphic coordinate systems are also para-holomorphic. Equivalently, if $(\widetilde{z}^b)$, $b=1,\ldots n$ is a para-holomorphic coordinate system on $U_\beta$, with $\widetilde{z}^b=\widetilde{x}^b+e\widetilde{y}^b$, then the following para-Cauchy-Riemann equations hold (see for instance \cite{Ka}):
\begin{equation}
\label{para-CR}
\frac{\partial \widetilde{x}^b}{\partial x^a}=\frac{\partial \widetilde{y}^b}{\partial y^a}\,,\,\frac{\partial \widetilde{x}^b}{\partial y^a}=\frac{\partial \widetilde{y}^b}{\partial x^a}\,,\,a,b\in\{1,\ldots n\}.
\end{equation}
In this case, on each $U_\alpha$, $I$ is given by
\begin{equation}
\label{I}
I(\frac{\partial}{\partial x^a})=\frac{\partial}{\partial y^a}\,,\,I(\frac{\partial}{\partial y^a})=\frac{\partial}{\partial x^a}.
\end{equation}

Now, we consider the para-complexification of the tangent bundle $TM$ as the $\mathbb{R}$-tensor product $T_CM=TM\otimes_{\mathbb{R}}C$ and its decomposition $T_CM=T^{1,0}M\oplus T^{0,1}M$ produced by $C$-linear extension of $I$ to $T_CM$, where
\begin{displaymath}
T_x^{1,0}M=\{Z=T_{x,C}M\,|\,IZ=eZ\}=\{X+eIX\,|\,X\in T_xM\},
\end{displaymath}
\begin{displaymath}
T_x^{0,1}M=\{Z=T_{x,C}M\,|\,IZ=-eZ\}=\{X-eIX\,|\,X\in T_xM\},
\end{displaymath}
are the eigenspaces of $I$ with eigenvalues $\pm e$. Also, if $I$ is integrable, that is $(M,I)$ is a para-complex manifold, the para-complex vectors
\begin{displaymath}
\frac{\partial}{\partial z^a}=\frac{1}{2}\left(\frac{\partial}{\partial x^a}+e\frac{\partial}{\partial y^a}\right)\,,\,\frac{\partial}{\partial z^{\overline{a}}}=\frac{1}{2}\left(\frac{\partial}{\partial x^a}-e\frac{\partial}{\partial y^a}\right)
\end{displaymath}
form a basis of the spaces $T^{1,0}_xM$ and $T^{0,1}_xM$.
\begin{remark}
(\cite{Cor1}) A $C$-valued function $f:M\rightarrow C$ on a para-complex manifold $(M,I)$ is para-holomorphic iff it satisfies 
\begin{equation}
\label{I5}
\frac{\partial f}{\partial z^{\overline{a}}}=0\,,\,\,\forall \,a\in\{1,\ldots,n\},
\end{equation}
where $(z^a)$ are local para-holomorphic coordinates on $(M,I)$ and $z^{\overline{a}}=\overline{z}^a$.
\end{remark}

\section{Para-complex Riemannian manifolds}
\setcounter{equation}{0}
Let $M$ be a para-complex manifold of para-complex dimension $n$ and denote by $(M,I)$ the manifold considered as a real $2n$-dimensional manifold with the induced almost para-complex structure $I$.
\begin{definition}
A \textit{para-complex Riemannian metric} on $M$ is a covariant symmetric $2$-tensor field $G:\Gamma(T_CM)\times\Gamma(T_CM)\rightarrow C$, which is non-degenerate at each point of $M$ and satisfies
\begin{equation}
\label{m1}
G(\overline{Z}_1,\overline{Z}_2)=\overline{G(Z_1,Z_2)}\,\,\,{\rm for\,\,every}\,\,Z_1,Z_2\in\Gamma(T_CM),
\end{equation}
\begin{equation}
\label{m2}
G(Z_1,Z_2)=0\,\,\,{\rm for\,\,every}\,\,Z_1\in\Gamma(T^{1,0}M)\,\,{\rm and}\,\,Z_2\in\Gamma(T^{0,1}M).
\end{equation}
\end{definition}
It is easy to see that the relation \eqref{m2} is equivalent to
\begin{equation}
\label{m3}
G(IZ_1,IZ_2)=G(Z_1,Z_2)\,\,\,{\rm for\,\,every}\,\,Z_1,Z_2\in\Gamma(T_CM),
\end{equation}
where we have denoted again by $I$  the $C$-linear extension of $I$ to $T_CM$. Thus a para-complex Riemannian metric on $M$ is completely determined by its values on $\Gamma(T^{1,0}M)$.
\begin{definition}
The pair $(M,G)$ consisting by a para-complex manifold $M$ and a para-complex Riemannian metric $G$ on $M$, will be called a \textit{para-complex Riemannian manifold}.
\end{definition}
If  $(z^a)$, $a=1,\ldots n$ is a para-holomorphic coordinate system on $M$, such that $\Gamma(T_CM)=span\{\partial/\partial z^a, \partial/\partial z^{\overline{a}}\}$, we put
\begin{equation}
\label{m4}
G_{AB}=G\left(\frac{\partial}{\partial z^A},\frac{\partial}{\partial z^B}\right)\,,\,\,A,B\in\{1,\ldots,n,\overline{1},\ldots,\overline{n}\}.
\end{equation}
Then, for a para-complex Riemannian metric $G$, the defining conditions \eqref{m1} and \eqref{m2} can be expressed locally as 
\begin{equation}
\label{m5}
G_{\overline{A}\,\overline{B}}=\overline{G_{AB}}\,,\,G_{a\overline{b}}=G_{\overline{a}b}=0.
\end{equation}
\begin{definition}
A para-complex Riemannian metric $G$ on a para-complex manifold $M$ is called \textit{para-holomorphic Riemannian metric} if the local components $G_{ab}$ are para-holomorphic functions, i.e.
\begin{equation}
\label{m6}
\frac{\partial G_{ab}}{\partial z^{\overline{c}}}=0\,,\,\,{\rm for\,\,every}\,\,c\in\{1\ldots,n\}.
\end{equation}
In this case, the pair $(M,G)$ is called a \textit{para-holomorphic Riemannian} manifold.
\end{definition}
As in the case of complex Riemannian manifolds (see \cite{Ga-Iv1}) for a given para-complex metric $G$ on $M$, we define the tensor field $\widetilde{G}$ on $M$ by setting
\begin{equation}
\label{m7}
\widetilde{G}(Z_1,Z_2)=(G\circ I)(Z_1,Z_2):=G(IZ_1,Z_2)\,\,\,{\rm for\,\,every}\,\,Z_1,Z_2\in\Gamma(T_CM).
\end{equation}
This metric is said to be associated with $G$ (also called \textit{twin metric}), and locally, it satisfies
\begin{equation}
\label{m8}
\widetilde{G}_{ab}=eG_{ab}\,\,{\rm and}\,\,\widetilde{G}_{\overline{a}\,\overline{b}}=-eG_{\overline{a}\,\overline{b}}.
\end{equation}
Also, we notice that given a para-complex Riemannian manifold $(M,G)$, the para-complex Riemannian metric $G$ induces a real Riemannian metric $g$ on the real manifold $(M,I)$ by setting
\begin{equation}
\label{a1}
g(X,Y)=2Re\,G(\widehat{X},\widehat{Y})\,,\,X,Y\in\Gamma(TM),
\end{equation}
where $\widehat{X}=(1/2)(X+eIX),\widehat{Y}=(1/2)(Y+eIY)\in\Gamma(T^{1,0}M)$, which satisfies
\begin{equation}
\label{m9}
g(IX,IY)=g(X,Y)\,\,{\rm or,\,\,equivalently},\,\,g(IX,Y)=g(X,IY)\,\,{\rm for\,\,every}\,\,X,Y\in\Gamma(TM).
\end{equation}
Such a real metric is also known as an almost product Riemannain metric or para-Norden metric, and the para-Norden manifold $(M,I,g)$ will be called the \textit{realization} of $(M,G)$. The para-Norden manifolds are studied for instance in \cite{S-I-A, S-I-E}. 

Conversely,  every para-Norden metric on the real manifold $(M,I)$ induces a para-complex Riemannian metric on the para-complex manifold $M$ by setting
\begin{equation}
\label{a2}
G(\widehat{X},\widehat{Y})=\frac{1}{2}\left(g(X,Y)+eg(X,IY)\right),\,X,Y\in\Gamma(TM)
\end{equation}
and $\widehat{X}=(1/2)(X+eIX),\widehat{Y}=(1/2)(Y+eIY)\in\Gamma(T^{1,0}M)$ as above, and next we extend $G$ to have the conditions \eqref{m1} and \eqref{m2}, which is possible because of \eqref{m9}.

Given any linear connection $D$ on a para-complex manifold $(M,I)$, with respect to a para-holomorphic coordinate system, we put 
\begin{displaymath}
D_{\frac{\partial}{\partial z^A}}\frac{\partial}{\partial z^B}=L^C_{AB}\frac{\partial}{\partial z^C}.
\end{displaymath}
We notice that the covariant differentiation, which is defined for real vector fields in $\Gamma(TM)$, can be extended by para-complex linearity on para-complex vector fields from $\Gamma(T_CM)$. Then $L^{\overline{C}}_{\overline{A}\,\overline{B}}=\overline{L^C_{AB}}$, where $\overline{\overline{A}}=A$. 

\begin{definition}
A (real) linear connection $D$ on $(M,I)$ is called \textit{almost para-complex} if $DI=0$.
\end{definition}
\begin{definition}
A para-Norden manifold $(M,I,g)$ is called \textit{para-K\"{a}hler Norden} manifold if the Levi-Civita connection $\nabla$ of $g$ is almost para-complex.
\end{definition}
Similary to the complex case \cite{O1,O2,S} (see also, \cite{B-F-F-V, B-F-V}), we have the following one-to-one corresponce between the para-K\"{a}hler Norden metrics and para-holomorphic Riemannian metrics on a para-complex manifold $(M,I)$.
\begin{proposition}
Let $(M,I)$ be a para-complex manifold. If $G$ is a para-holomorphic Riemannian metric $(M,I)$ then $g$ defined in \eqref{a1} is a para-K\"{a}hler Norden metric on $(M,I)$, and conversely if $g$ is a para-K\"{a}hler metric on the (real) manifold $(M,I)$ then $G$ defined in \eqref{a2} is a para-holomorphic Riemannian metric on $(M,I)$.
\end{proposition}
By direct calculus, we easy obtain
\begin{proposition}
A linear connection $D$ on $M$ is almost para-complex iff $L^{\overline{c}}_{ab}=L^c_{a\overline{b}}=0$ 
\end{proposition}
Now, let us denote by $\nabla$ and $\widetilde{\nabla}$ the Levi-Civita connections of $G$ and $\widetilde{G}$, respectively. Then, as usual, the Christoffel symbols of $G$ are given by
\begin{equation}
\label{m10}
\Gamma^C_{AB}=\frac{1}{2}G^{CD}\left(\frac{\partial G_{BD}}{\partial z^A}+\frac{\partial G_{AD}}{\partial z^B}-\frac{\partial G_{AB}}{\partial z^D}\right),
\end{equation}
where $(G^{AB})_{n\times n}$ denotes the inverse matrix of $(G_{AB})_{n\times n}$, and similarly for the Christoffel symbols $\widetilde{\Gamma}^C_{AB}$ of $\widetilde{G}$. 

Taking into account \eqref{m5} and \eqref{m8}, we have the following relations which relates the Christoffel symbols of $G$ and $\widetilde{G}$, respectively
\begin{equation}
\label{m12}
\widetilde{\Gamma}^c_{ab}=\Gamma^c_{ab}=\frac{1}{2}G^{cd}\left(\frac{\partial G_{bd}}{\partial z^a}+\frac{\partial G_{ad}}{\partial z^b}-\frac{\partial G_{ab}}{\partial z^d}\right)
\end{equation}
\begin{equation}
\label{m13}
\widetilde{\Gamma}^{\overline{c}}_{ab}=-\Gamma^{\overline{c}}_{ab}=\frac{1}{2}G^{\overline{c}\,\overline{d}}\frac{\partial G_{ab}}{\partial z^{\overline{d}}}\,\,,\,\,\widetilde{\Gamma}^{c}_{\overline{a}b}=\Gamma^{c}_{\overline{a}b}=\frac{1}{2}G^{cd}\frac{\partial G_{bd}}{\partial z^{\overline{a}}}.
\end{equation}
By analogy with the complex case \cite{Ga-Iv1}, we define the fundamental  tensor $\Phi$ on a para-complex Riemannian manifold by setting
\begin{equation}
\label{m14}
\Phi(Z_1,Z_2)=\widetilde{\nabla}_{Z_1}Z_2-\nabla_{Z_1}Z_2\,,\,\,{\rm for\,\,every}\,\,Z_1,Z_2\in\Gamma(T_CM).
\end{equation}
By this definition, we deduce
\begin{equation}
\label{m15}
\Phi(\overline{Z}_1,\overline{Z}_2)=\overline{\Phi(Z_1,Z_2)}\,,\,\,{\rm for\,\,every}\,\,Z_1,Z_2\in\Gamma(T_CM).
\end{equation} 
Using \eqref{m14}, \eqref{m12}, \eqref{m13} and \eqref{m15} the nonvanishing components of the fundamental tensor $\Phi$ are given by
\begin{equation}
\label{m16}
\Phi^{\overline{c}}_{ab}=G^{\overline{c}\,\overline{d}}\frac{\partial G_{ab}}{\partial z^{\overline{d}}}\,\,{\rm and}\,\,\Phi^c_{\overline{a}\,\overline{b}}=\overline{\Phi^{\overline{c}}_{ab}}.
\end{equation}
Also, from \eqref{m14} and \eqref{m16} we have
\begin{proposition}
The fundamental tensor of a para-complex Riemannian manifold $(M,G)$ satisfy
\begin{equation}
\label{m17}
\Phi(Z_1,Z_2)=\Phi(Z_2,Z_1)\,,\,\Phi(IZ_1,Z_2)=-I\Phi(Z_1,Z_2)\,,\,\,\forall\,Z_1,Z_2\in\Gamma(T_CM).
\end{equation}
\end{proposition}
\begin{remark}
If $(M,I,g)$ is the realization of a para-complex Riemannian manifold $(M,G)$ we can define as in \eqref{m14} the fundamental tensor for real vector fields, and the property \eqref{m15} of $\Phi$ implies that $\Phi$ is the para-complex extension of the real fundamental tensor on $(M,I,g)$.
\end{remark}

In the following, we extend the study from \cite{Ga-Iv1} to the para-complex case, and we shall construct a \textit{characteristic} linear connection on a para-complex Riemannian manifold.

We consider the fundamental tensor of type $(0,3)$ defined by
\begin{equation}
\label{18}
\Psi(Z_1,Z_2,Z_3)=G(\Phi(Z_1,Z_2),Z_3)\,,\,\,{\rm for\,\,every}\,\,Z_1,Z_2,Z_3\in\Gamma(T_CM).
\end{equation}
In a para-holomorphic coordinate system on $M$, we have locally
\begin{equation}
\label{m19}
\Psi_{AB,C}=\Phi^D_{AB}G_{DC},
\end{equation}
and the nonvanishing componets of $\Psi_{AB,C}$ are 
\begin{equation}
\label{m20}
\Psi_{ab,\overline{c}}=\frac{\partial G_{ab}}{\partial z^{\overline{c}}}\,\,{\rm and}\,\,\Psi_{\overline{a}\,\overline{b},c}=\overline{\Psi_{ab,\overline{c}}}.
\end{equation}
We have
\begin{theorem}
\label{tm1}
On every para-complex Riemannian manifold $(M,G)$ there exists a unique linear connection $D$ with local coefficients $L^C_{AB}$ such that
\begin{enumerate}
\item[(i)] $D$ is symmetric, that is $L^C_{AB}=L^C_{BA}$;
\item[(ii)] $D$ is almost para-complex, that is $L^{\overline{c}}_{ab}=L^c_{a\overline{b}}=0$;
\item[(iii)] The covariant derivatives $D_aG_{bc}=\partial G_{bc}/\partial z^a-L^d_{ab}G_{dc}-L^d_{ac}G_{bd}$ vanishes.
\end{enumerate}
\end{theorem}
\begin{proof}
If we define the local coefficients of $D$ by
\begin{equation}
\label{m22}
L^C_{AB}=\Gamma^C_{AB}+\frac{1}{2}\Phi^C_{AB}-\frac{1}{2}G^{CD}(\Psi_{DA,B}+\Psi_{DB,A}), 
\end{equation}
where $\Gamma^C_{AB}$ are the para-complex Christoffel symbols of $G$, then by direct calculus we obtain that $D$ satisfies the conditions of theorem.

Also, if $D^\prime$  is another connection which satisfy the all conditions of theorem, with local coefficients $L^{\prime C}_{AB}$, we denote by $D^C_{AB}=L^C_{AB}-L^{\prime C}_{AB}$ the difference tensor. Then, we easily obtain
\begin{equation}
\label{m23}
D^C_{AB}=D^C_{BA}\,,\,D^{\overline{c}}_{ab}=D^{c}_{a\overline{b}}=0\,,\,D^d_{ab}G_{dc}+D^d_{ac}G_{ab}=0,
\end{equation}
which implies $D^C_{AB}=0$, that is $D=D^\prime$, and the uniqueness then follows.
\end{proof}
The linear connection from the above theorem, will be called the \textit{characteristic connection} of the para-complex Riemannian manifold $(M,G)$.

The defining equality \eqref{m22} of the characteristic connection and the properties of the fundamental tensor, implies
\begin{corollary}
\label{cm1}
On every para-complex Riemannian manifold $(M,G)$ there exists a unique linear connection $D$ such that
\begin{enumerate}
\item[(i)] $D$ is symmetric;
\item[(ii)] $D$ is almost para-complex;
\item[(iii)]  $D_AG_{BC}=\Psi_{BC,A}$, i.e the covariant derivative of the metric $G$ is the fundamental tensor $\Psi$.
\end{enumerate}
\end{corollary}
\begin{remark}
The third condition of Theorem \ref{tm1} says that the nonvanishing components of the tensor $D_AG_{BC}$ are 
\begin{equation}
\label{m24} 
D_{\overline{a}}G_{bc}=\Psi_{bc,\overline{a}}\,\,{\rm and}\,\,D_{a}G_{\overline{b}\,\overline{c}}=\overline{D_{\overline{a}}G_{bc}}.
\end{equation}
\end{remark}
On the realization of a para-complex Riemannian manifold we have
\begin{corollary}
If $(M,I,g)$ is the realization of a para-complex Riemannian manifold $(M,G)$, then the characteristic connection $D$ on $(M,I,g)$ is the unique connection which satisfy the conditions
\begin{enumerate}
\item[(i)] $D$ is symmetric;
\item[(ii)] $D$ is almost para-complex;
\item[(iii)] $(D_{X}g)(Y,Z)=(D_{IX}g)(IY,Z)$, for every $X,Y,Z\in\Gamma(TM)$.
\end{enumerate}
\end{corollary}
The defining equality \eqref{m22} and \eqref{m20} imply that the nonvanishing coefficients of the caracteristic connection $D$ are
\begin{equation}
\label{m25}
L^c_{ab}=\Gamma^c_{ab}\,\,{\rm and}\,\,L^{\overline{c}}_{\overline{a}\,\overline{b}}=\overline{L^c_{ab}},
\end{equation}
that is, $D$ is completely determined on $\Gamma(T^{1,0}M)$.

We notice that a vector field $Z=Z^a(\partial/\partial z^a)$ on a para-complex manifold is para-holomorphic if $Z^a$ are para-holomorphic functions. Also, according to Lemma 2.1.6 \cite{Kr}, a vector field $\widehat{X}=(1/2)(X+eIX)$ is para-holomorphic iff 
\begin{equation}
\label{y3}
(\mathcal{L}_XI)Y=[X,IY]-I[X,Y]=0\,,\,\forall\,Y\in\Gamma(TM).
\end{equation}
In that follows we denote the set of para-holomorphic vector fields on $(M,I)$ by $\Gamma_{ph}(T^{1,0}M)$. 
\begin{definition}
A linear connection $D$ on $M$ is called \textit{para-holomorphic} if $D_{Z_1}Z_2\in \Gamma_{ph}(T^{1,0}M)$ for arbitrary para-holomorphic vector fields $Z_1,Z_2$. 
\end{definition}
We have
\begin{proposition}
The caracteristic connection $D$ of a para-complex Riemannian manifold $(M,G)$ is para-holomorphic iff the para-complex Christoffel symbols $L^c_{ab}=\Gamma^c_{ab}$ are para-holomorphic functions.
\end{proposition}
As a direct consequence of \eqref{m20}, \eqref{m12}, \eqref{m13}, Corollary \ref{cm1} and \eqref{m22}, we get
\begin{theorem}
\label{tm2}
For every para-complex Riemannian manifold $(M,G)$, the following assertions are equivalent:
\begin{enumerate}
\item[(i)] The fundamental tensor $\Phi$ (or $\Psi$) is zero;
\item[(ii)] The local components $G_{ab}$ of the metric $G$ are para-holomorphic functions;
\item[(iii)] The Levi-Civita connection $\nabla$ of $G$ is almost para-complex, that is $\nabla I=0$;
\item[(iv)] The characteristic connection $D$ is metrical with respect to $G$, that is $DG=0$;
\item[(v)] The Levi-Civita connection $\nabla$ coincides with the characteristic connection $D$.
\end{enumerate} 
\end{theorem}
Let  $R$ be the characteristic curvature  tensor of the characteristic connection $D$, defined as usual by
\begin{displaymath}
R(X,Y)Z=\left[D_{X},D_{Y}\right]Z-D_{[X,Y]}Z,\,\,{\rm  for \,\,every}\,\,X,Y,Z\in\Gamma(T_CM).
\end{displaymath} 
The local components of $R$ are given by
\begin{equation}
\label{m26}
R\left(\frac{\partial}{\partial z^A},\frac{\partial}{\partial z^B}\right)\frac{\partial}{\partial z^C}=R^D_{C, AB}\frac{\partial}{\partial z^D},
\end{equation}
and the nonvanishing components of $R$ are
\begin{equation}
\label{m27}
R^d_{c, ab}=\frac{\partial L^d_{cb}}{\partial z^a}-\frac{\partial L^d_{ca}}{\partial z^b}+L^f_{cb}L^d_{fa}-L^f_{ca}L^d_{fb}\,,\,R^{\overline{d}}_{\overline{c}, \overline{a}\,\overline{b}}=\overline{R^d_{c, ab}},
\end{equation}
\begin{equation}
\label{m28}
R^d_{c,\overline{a}b}=\frac{\partial L^d_{bc}}{\partial z^{\overline{a}}}\,,\,R^{\overline{d}}_{\overline{c},a\overline{b}}=\overline{R^d_{c,\overline{a}b}}.
\end{equation}
It is easy to see that $R^d_{c,\overline{a}b}=0$ if and only if $D$ is a para-holomorphic connection. Also, the characteristic Riemann curvature tensor of $D$ is defined as usual by $\mathcal{R}(Z_1,Z_2,Z_3,Z_4)=G(R(Z_1,Z_2)Z_3, Z_4)$ and its local components are $R_{ABCD}=G_{DF}R^F_{C,AB}$. Its nonvanishing components are 
\begin{equation}
\label{31}R_{abcd}=G_{df}R^f_{c,ab}\,\,{\rm and}\,\,R_{\overline{a}bcd}=G_{df}R^f_{c,\overline{a}b},
\end{equation}
and their para-complex conjugates.

Moreover, every nondegenerate $2$-plane in $T^{1,0}_zM$ is called a \textit{para-holomorphic} $2$-plane, and the \textit{para-holomorphic} characteristic sectional curvature for a given $2$-plan $P=span\{Z_1,Z_2\}$, where $Z_1,Z_2\in\Gamma(T_z^{1,0}M)$, $z\in M$ is defined by
\begin{equation}
\label{m30}
K_z(P)=\frac{\mathcal{R}(Z_1,Z_2,Z_1,Z_2)}{G(Z_1,Z_1)G(Z_2,Z_2)-(G(Z_1,Z_2))^2}.
\end{equation}
The following Schur type theorem holds.
\begin{theorem}
Let $(M,G)$ be a connected para-holomorphic Riemannian manifold of para-complex dimension $n\geq3$. If the para-holomorphic sectional curvatures does not depend on the $2$-plane $P$, then $(M,G)$ is of constant para-holomorphic sectional curvature.
\end{theorem}

In the end of this section we describe the Einstein equations for para-complex Riemannian manifolds. The associated characteristic Ricci tensor ${\rm Ric}$ is locally given by 
\begin{equation}
\label{y1}
{\rm Ric}\left(\frac{\partial}{\partial z^C},\frac{\partial}{\partial z^A}\right)={\rm Ric}_{CA}=R^B_{C,AB},
\end{equation}
and its nonvanishing components are 
\begin{equation}
\label{m29}
{\rm Ric}_{ca}=R^b_{c,ab}\,,\,{\rm Ric}_{c\overline{a}}=R^b_{c,\overline{a}b}\,,\,{\rm Ric}_{\overline{c}\,\overline{a}}=\overline{{\rm Ric}_{ca}}\,,\,{\rm Ric}_{\overline{c}a}=\overline{{\rm Ric}_{c\overline{a}}}.
\end{equation}
The function $\rho$ defined by 
\begin{equation}
\label{y2}
\rho=G^{CA}{\rm Ric}_{CA}=G^{ca}{\rm Ric}_{ca}+G^{\overline{c}\,\overline{a}}{\rm Ric}_{\overline{c}\,\overline{a}}
\end{equation}
is called the \textit{scalar curvature} of $D$ and it is a real function. 

The equation
\begin{equation}
\label{III3}
{\rm Ric}-\frac{\rho}{2} G=8\pi cT
\end{equation}
is called the \textit{Einstein equation} of the para-complex Riemannian  manifold $(M,G)$.   In the equation \eqref{III3}, the left hand side is called the \textit{Einstein curvature} which is constructed using the para-complex Riemannian metric $G$, while in the right hand side we have a tensor $T$ called the \textit{stress-energy-momentum tensor} and represents the matter and energy that generate the gravitational field of potentials $(G_{AB})$. The constant $c$ is the gravitational constant. Locally, the  Einstein equation is expressed as 
\begin{equation}
\label{III4}
{\rm Ric}_{AB}-\frac{\rho}{2} G_{AB}=8\pi c T_{AB}.
\end{equation}

If the Einstein equation holds, then taking into account \eqref{m5} it follows that ${\rm Ric}_{a\overline{b}}=8\pi cT_{a\overline{b}}$. In the empty leave space (no matter, no energy) we have $T_{AB}=0$, and contracting \eqref{III4} with $G^{AB}$ one gets $\rho=0$ and so it reduced to
\begin{equation}
\label{III5}
{\rm Ric}_{AB}=0.
\end{equation} 
Consequently, ${\rm Ric}_{ab}={\rm Ric}_{\overline{a}\,\overline{b}}=0$. 

Letting $E_{AB}={\rm Ric}_{AB}-(\rho/2) G_{AB}$ and  $E^A_B=G^{AC}E_{CB}$, the divergence of $E$ is defined by
\begin{equation}
\label{III6}
{\rm div}\,E=E^A_{B|A},
\end{equation}
where "$|$" denotes the covariant derivative with respect to $\nabla$ and we have ${\rm div}\,E=0$. The proof is based on the second  Bianchi identity $\sum\limits_{cycl}(\nabla_{X}R)(Y,Z)=0$ written in a local basis $\left\{\partial/\partial z^A\right\}$ of $\Gamma(T_CM)$. Assuming the Einstein equation holds, by using ${\rm div}\,E=0$, we must have
\begin{equation}
\label{III7}
{\rm div}\,T=0,
\end{equation}
which is called the \textit{continuity condition} for para-complex Riemannian manifold $(M,G)$. 

Finally, by analogy with the complex case, see \cite{Iv}, the following result concerning the Einstein condition for para-complex Riemannian manifolds holds. 
\begin{definition}
The para-complex Riemannian manifold $(M,G)$ is said to be \textit{characteristic Einstein} if ${\rm Ric}_{c\overline{a}}=0$ and ${\rm Ric}_{ca}=f G_{ca}$, where $f=f_1+ef_2$ is a para-complex valued function on $M$. 
\end{definition}
\begin{theorem}
Let $(M,G)$ be a $m$-dimensional para-complex Riemannian characteristic Einstein manifold with $m\geq 3$. Then the characteristic scalar curvature $\rho_{0}=G^{ca}{\rm Ric}_{ca}$ is an anti-para-holomorphic function on $M$ and ${\rm Ric}_{ca}=(\rho_{0}/m)G_{ca}$.
\end{theorem}

\section{Para-holomorphic Riemannain Einstein manifolds}
\setcounter{equation}{0}

We recall that a (real) metric $g$ on the (real) manifold $M$ is said to be an \textit{Einstein metric} if
\begin{equation}
\label{e1}
{\rm Ric}(g)=\lambda g,
\end{equation}
where $\lambda$ is a real constant and ${\rm Ric}(g)$ denotes the Ricci tensor of the metric $g$.

In this section, we prove that by taking the real part of a para-holomorphic Einstein metric on a para-complex manifold $(M,I)$ of para-complex dimension $n$ one gets a real Einstein manifold of real dimension $2n$ obtaining a result similar to Theorem 5.1 from \cite{B-F-V} from the anti-K\"{a}hlerian manifolds case. 

Let $(M,G)$ be a para-holomorphic Riemannian manifold. Then, as we already noticed in the previous section the relations \eqref{a1} and \eqref{a2} establishes an one-to-one correspondence between the para-K\"{a}hler Norden metrics on the (real) manifold $(M,I)$ and the para-holomorphic metrics on the para-complex manifold $M$.

Although we can follow an argument similar from \cite{B-F-F-V, B-F-V}, for a better presentation of the notions that we use, in this section we denote the para-holomorphic Riemannain metric $G$ by $\widehat{g}$ and we follow an argument similar to \cite{O2, S} for K\"{a}hler-Norden manifolds.

Without loss of generality, we consider the (real) vector fields $X,Y,\ldots \in\Gamma(TM)$ such that $\widehat{X},\widehat{Y},\ldots \in\Gamma_{ph}(T^{1,0}M)$, are para-holomorphic vector fields on the para-complex manifold $(M,I)$, that is the relation \eqref{y3} holds. Then, we easily obtain
\begin{equation}
\label{e2}
[IX,Y]=[X,IY]=I[X,Y]\,,\,[IX,IY]=[X,Y]\,,\,[\widehat{X},\widehat{Y}]=\widehat{[X,Y]}=:[X,Y]^{\,\widehat{}}.
\end{equation}
Also, by a direct calculation, we have that for every para-complex function $f=Re\,f+eIm\,f$ on $M$, and every vector field $X\in\Gamma(TM)$, the following relation holds 
\begin{equation}
\label{e3}
f\widehat{X}=((Re\,f)X+(Im\,f)X)^{\,\widehat{}},
\end{equation}
and, moreover, if $f$ is para-holomorphic, then the para-Cauchy-Riemann equations imply
\begin{equation}
\label{e4}
X(Re\,f)=(IX)(Im\,f)\,,\,(IX)(Re\,f)=X(Im\,f)\,,\,\widehat{X}f=X(Re\,f)+eX(Im\,f).
\end{equation}
Now, for every real tangent space $T_{z,\mathbb{R}}M$, $z\in M$, we can choose an adapted orthonormal (real) frame $\{e_a,Ie_a\}$, $a\in\{1,\ldots,n\}$, such that
\begin{equation}
\label{e5}
g(e_a,e_b)=g(Ie_a,Ie_b)=\delta_{ab}\,,\,g(e_a,Ie_b)=0\,,\,a,b\in\{1,\ldots,n\}.
\end{equation}  
Then, we obtain an adapted para-complex frame $\{\widehat{e_a}\}$, $a\in\{1,\ldots,n\}$, for $\Gamma(T_z^{1,0}M)$, where $\widehat{e}_a=(1/2)(e_a+eIe_a)$ for which $\widehat{g}(\widehat{e}_a,\widehat{e}_b)=(1/2)\delta_{ab}$.

Let $\nabla$ and $\widehat{\nabla}$ be the Levi-Civita connections of the para-K\"{a}hler Norden metric $g$ and of the para-holomorphic Riemannian metric $\widehat{g}$, respectively. According to the discussion from the previous section, $\widehat{\nabla}$ is a para-holomorphic connection, and also, by the symmetry of $\nabla$ and using \eqref{e2}, we obtain
\begin{equation}
\label{e6}
\nabla_{IX}Y=I\nabla_XY\,,\,\,\forall\,X,Y\in\Gamma(TM).
\end{equation}
Let us consider now the Koszul formula which gives the Levi-Civita connections $\widehat{\nabla}$ of $\widehat{g}$
\begin{equation}
\begin{array}{ll}
2\widehat{g}(\widehat{\nabla}_{\widehat{X}}\widehat{Y},\widehat{Z})=\widehat{X}(\widehat{g}(\widehat{Y},\widehat{Z}))+\widehat{Y}(\widehat{g}(\widehat{X},\widehat{Z}))-\widehat{Z}(\widehat{g}(\widehat{X},\widehat{Y})) \\
\,\,\,\,\,\,\,\,\,\,\,\,\,\,\,\,\,\,\,\,\,\,\,\,\,\,\,\,\,\,\,\,\,\,\,\,\,\,\,-\widehat{g}([\widehat{X},\widehat{Z}],\widehat{Y})-\widehat{g}([\widehat{Y},\widehat{Z}],\widehat{X})+\widehat{g}([\widehat{X},\widehat{Y}],\widehat{Z}).
\end{array}
\label{e7}
\end{equation}
and similar, we can write this formula for the real metric $g$, with $\nabla$, $g$, $X,Y$ and $Z$, respectively.

Using \eqref{a2}, \eqref{e2} and \eqref{e4}, it follows
\begin{eqnarray*}
&&\widehat{X}\widehat{g}(\widehat{Y},\widehat{Z})=\frac{1}{2}(Xg(Y,Z)+eXg(Y,IZ)),\\
&&\widehat{g}([\widehat{X},\widehat{Y}],\widehat{Z})=\frac{1}{2}(g([X,Y],Z)+eg([X,Y],IZ)),\\
&&Zg(X,IY)=(IZ)g(X,Y).
\end{eqnarray*}
Now, by the above formulas, \eqref{e2} and the Koszul formula \eqref{e7} for the real metric $g$, the relation \eqref{e7} becomes
\begin{equation}
\label{e8}
2\widehat{g}(\widehat{\nabla}_{\widehat{X}}\widehat{Y},\widehat{Z})=g(\nabla_XY,Z)+eg(\nabla_XY,IZ)=2\widehat{g}(\widehat{\nabla_{X}Y},\widehat{Z}),
\end{equation}
which implies the following important relation
\begin{equation}
\label{e9}
\widehat{\nabla}_{\widehat{X}}\widehat{Y}=\widehat{\nabla_{X}Y}.
\end{equation}
In the sequel we consider the Riemann curvature tensors $R$ and $\widehat{R}$ of $\nabla$ and $\widehat{\nabla}$, respectively. Taking into account that $\nabla$ is almost para-complex, i.e $\nabla I=0$, and also using \eqref{e2} and \eqref{e6}, we obtain that $R$ is totally pure (or $I$-symmetric), that is (see also \cite{S-I-E})
\begin{equation}
\label{e10}
R(X,Y)I=R(IX,Y)=R(X,IY)=IR(X,Y).
\end{equation} 
By direct calculus, using \eqref{e2} and \eqref{e9}, it follows that the Riemann curvature tensors $R$ and $\widehat{R}$ are related by
\begin{equation}
\label{e11}
\widehat{R}(\widehat{X},\widehat{Y})\widehat{Z}=(R(X,Y)Z)^{\,\widehat{}}.
\end{equation}
Now, let us consider the Ricci tensor fields associated to the metrics $g$ and $\widehat{g}$, respectively, given by
\begin{equation}
\label{e12}
{\rm Ric}(g)(X,Y)={\rm Tr}\{Z\mapsto R(Z,X)Y\}\,\,{\rm and}\,\,{\rm Ric}(\widehat{g})(\widehat{X},\widehat{Y})={\rm Tr}\{\widehat{Z}\mapsto\widehat{R}(\widehat{Z},\widehat{X})\widehat{Y}\},
\end{equation}
and let us denote by $Q$ and $\widehat{Q}$ be the associated Ricci operators, given by
\begin{equation}
\label{e13}
g(QX,Y)={\rm Ric}(g)(X,Y)\,\,{\rm and}\,\,\widehat{g}(\widehat{Q}\widehat{X},\widehat{Y})={\rm Ric}(\widehat{g})(\widehat{X},\widehat{Y}).
\end{equation} 
We have 
\begin{proposition}The Ricci tensors ${\rm Ric}(g)$, ${\rm Ric}(\widehat{g})$ and the Ricci operators $Q$, $\widehat{Q}$ satisfy the following relations
\begin{equation}
\label{e14}
{\rm Ric}(g)(IX,Y)={\rm Ric}(g)(X,IY)\,,\,{\rm Ric}(g)(IX,IY)={\rm Ric}(g)(X,Y)\,,\,QI=IQ
\end{equation}
and
\begin{equation}
\label{e15}
{\rm Ric}(\widehat{g})(\widehat{X},\widehat{Y})=\frac{1}{2}({\rm Ric}(g)(X,Y)+e{\rm Ric}(g)(X,IY))\,,\,\widehat{Q}\widehat{X}=\widehat{QX}.
\end{equation}
\end{proposition} 
\begin{proof}
The relations \eqref{e14} follow directly from the defining relations \eqref{e12} and \eqref{e13} and using \eqref{e10}.

For the first relation of \eqref{e15}, using the orthonormal frame $\{e_a,Ie_a\}$, $a\in\{1,\ldots,n\}$, we have
\begin{displaymath}
{\rm Ric}(g)(X,Y)=\sum_a(g(R(e_a,X)Y,e_a)+g(R(Ie_a,X)Y,Ie_a))=2\sum_a(g(R(e_a,X)Y,e_a)),
\end{displaymath}
where we have also used \eqref{e10} and \eqref{m9}. Next, using the adapted para-complex frame $\{\widehat{e}_a\}$, $a\in\{1,\ldots,n\}$ and the formulas \eqref{a2}, \eqref{e10} and \eqref{e11} we obtain
\begin{eqnarray*}
{\rm Ric}(\widehat{g})(\widehat{X},\widehat{Y})&=&2\sum_a\widehat{g}(\widehat{R}(\widehat{e}_a,\widehat{X})\widehat{Y},\widehat{e}_a)=2\sum_a \widehat{g}((R(e_a,X)Y)^{\,\widehat{}},\widehat{e}_a)\\
&=&\sum_a(g(R(e_a,X)Y,e_a)+eg(R(e_a,X)Y,Ie_a))\\
&=&\sum_a(g(R(e_a,X)Y,e_a)+eg(R(e_a,X)IY,e_a)),
\end{eqnarray*}
which together with the previous equality implies the first relation of \eqref{e15}. This together with \eqref{a2} gives the following relation for the Ricci operators $Q$ and $\widehat{Q}$
\begin{eqnarray*}
\widehat{g}(\widehat{Q}\widehat{X},\widehat{Y})&=&{\rm Ric}(\widehat{g})(\widehat{X},\widehat{Y})=\frac{1}{2}({\rm Ric}(g)(X,Y)+e{\rm Ric}(g)(X,IY))\\
&=&\frac{1}{2}(g(QX,Y)+eg(QX,IY))=\widehat{g}(\widehat{QX},\widehat{Y}),
\end{eqnarray*}
which proves the second relation of \eqref{e15}.
\end{proof}
The first relation of \eqref{e15} leads to the announced result, that is
\begin{theorem}
\label{tIII1}
Let us suppose that $(M,I,g)$ is a para-K\"{a}hlerian Norden manifold, that is a para-complex  manifold of para-complex dimension $n$ endowed with a para-holomorphic Riemannian metric $\widehat{g}\equiv (\widehat{g}_{ab}(z))$, $a,b\in\{1,\ldots,n\}$ and with a real metric $g\equiv (g_{jk}(x))$, $j,k\in\{1,\ldots,2n\}$ given by $g=2Re\,\widehat{g}$. Then the para-holomorphic  metric $\widehat{g}$  is Einstein with the real constant $\lambda$ if and only if the real metric $g$ is Einstein metric with the same constant.
\end{theorem}
\begin{remark}
We notice that starting from the original para-K\"{a}hlerian Norden metric $g$ on a para-complex manifold $(M,I)$, the real twin metric can be considered, that is $h(X,Y):=(g\circ I)(X,Y)=g(IX,Y)$, for every $X,Y\in\Gamma(TM)$. We find
\begin{equation}
\label{III9}
h(X,Y)=2{\rm Im}\widehat{g}(\widehat{X},\widehat{Y})\,,\,\,\forall\,X,Y\in\Gamma(TM).
\end{equation}
Moreover, if we denote by $\nabla$ the covariant differentiation of the Levi-Civita connection associated to the para-K\"{a}hlerian Norden metric $g$, then we have (see \cite{S-I-E})
\begin{equation}
\label{III10}
\nabla h=\nabla g\circ I+g\circ\nabla I=0.
\end{equation}
The above relation says that, the Levi-Civita connection of $g$ coincides with the Levi-Civita connection of $h$, thus they have the same real and para-complex Riemann and Ricci tensors (see also the discussion from the previous section). In the real case only one of two twin metrics can be Einsteinian. In para-complex case the Einstein condition ${\rm Ric}(\widehat{g})=\lambda \widehat{g}$ implies ${\rm Ric}(\widehat{h})=e\lambda \widehat{h}$, that is , both para-holomorphic metrics $\widehat{g}$ and $\widehat{h}$ are Einstein metrics at the same time. Also, we can conclude that the metric $h$ is an Einstein metric with an imaginary cosmological constant.  
\end{remark}
If the para-holomorphic metric $\widehat{g}$ is Einstein with para-complex constant $\widehat{\lambda}$, that is
\begin{equation}
\label{e18}
{\rm Ric}(\widehat{g})=\widehat{\lambda}\widehat{g},\,\widehat{\lambda}\in C,
\end{equation} 
then, similarly to the K\"{a}hler-Norden manifolds from the complex case, see \cite{O2}, we can describe the following generalization of Theorem \ref{tIII1}.

We consider the real scalar curvatures $K, K^*$ of $g$, and the para-holomorphic scalar curvature $\widehat{K}$ of $\widehat{g}$, that is
\begin{displaymath}
K={\rm Tr}Q\,,\,K^*={\rm Tr}(IQ)\,,\,\widehat{K}={\rm Tr}(\widehat{Q}).
\end{displaymath}
We have
\begin{proposition}
The real scalar curvatures $K$, $K^*$ and the para-holomorphic scalar curvature $\widehat{K}$ are related by
\begin{equation}
\label{e16}
\widehat{K}=\frac{1}{2}(K+eK^*).
\end{equation}
\end{proposition}
\begin{proof}
Using \eqref{e14} and \eqref{m9}, we obtain the following expressions for $K$ and $K^*$:
\begin{displaymath}
K=\sum_a(g(Qe_a,e_a)+g(QIe_a,Ie_a))=2\sum_a(g(Qe_a,e_a)),
\end{displaymath}
and
\begin{displaymath}
K^*=\sum_a(g(IQe_a,e_a)+g(IQIe_a,Ie_a))=2\sum_a(g(Qe_a,Ie_a)).
\end{displaymath}
Now, using \eqref{e15} and \eqref{a2}, we obtain
\begin{displaymath}
\widehat{K}=2\sum_a\widehat{g}(\widehat{Q}\widehat{e}_a,\widehat{e}_a)=2\sum_a\widehat{g}(\widehat{Qe_a},\widehat{e}_a)=\sum_a(g(Qe_a,e_a)+eg(Qe_a,Ie_a))
\end{displaymath}
which proves \eqref{e16}.
\end{proof}
Now, by applying the para-Cauchy-Riemann equations to the para-holomorphic function $\widehat{K}$ and taking into account that $Re\,\widehat{K}=K/2$ and $Im\,\widehat{K}=K^*/2$, we get
\begin{displaymath}
dK(X)=XK=(IX)K^*=dK^*(IX)\,\,{\rm and}\,\,dK(IX)=(IX)K=XK^*=dK^*(X),
\end{displaymath}
which implies
\begin{equation}
\label{e17}
d\widehat{K}(\widehat{X})=\frac{1}{2}(dK(X)+edK(IX)).
\end{equation}
Then, the following theorem, which is an analogue of Theorem 1 from \cite{O2} for K\"{a}hler-Norden manifolds, holds.
\begin{theorem}
\label{tIII2}
 The para-holomorphic Riemannian manifold $(M,\widehat{g})$ is para-holomorphic Einstein with para-complex constant $\widehat{\lambda}=\lambda_1+e\lambda_2$ iff
\begin{equation}
\label{e19}
{\rm Ric}(g)(X,Y)=\lambda_1g(X,Y)+\lambda_2g(X,IY).
\end{equation}
Moreover, in the formula \eqref{e19}, we have $\lambda_1=K/2n$ and $\lambda_2=K^*/2n$.
\end{theorem}
\begin{proof}
Taking into account the formula \eqref{a2} and the first relation of \eqref{e15}, we see that \eqref{e18} holds iff
\begin{eqnarray*}
{\rm Ric}(g)(X,Y)+e{\rm Ric}(g)(X,IY)=\lambda_1g(X,Y)+\lambda_2g(X,IY)+e(\lambda_2g(X,Y)+\lambda_1g(X,IY)),
\end{eqnarray*}
which is equivalent to \eqref{e19}. Moreover, using \eqref{e19}, it follows that the Ricci operator $Q$ satisfy
\begin{displaymath}
QX=\lambda_1X+\lambda_2IX\,\,{\rm and}\,\,IQX=\lambda_1IX+\lambda_2X.
\end{displaymath}
Now, the shape constants $\lambda_1$, $\lambda_2$ can be obtained from these two relations using the definitions of $K$ and $K^*$, respectively. Indeed, using the orthonormal frame $\{e_a,Ie_a\}$, $a\in\{1,\ldots,n\}$, we have
\begin{eqnarray*}
K&=&{\rm Tr}Q=\sum_a(g(Qe_a,e_a)+g(QIe_a,Ie_a)\\
&=&\sum_a(g(\lambda_1e_a+\lambda_2Ie_a,e_a)+g(\lambda_1Ie_a+\lambda_2e_a,Ie_a))\\
&=&\sum_a(\lambda_1g(e_a,e_a)+\lambda_1g(Ie_a,Ie_a))=2n\lambda_1
\end{eqnarray*}
and similarly, we obtain $K^*=2n\lambda_2$.
\end{proof}

\section{Para-complex Lie groups as para-holomorphic Riemannian Einstein manifolds}
\setcounter{equation}{0}

\begin{definition} 
A \textit{para-complex Lie group}, is a group $G$, which is also a para-complex manifold, such that the group multiplication $\phi:G\times G\rightarrow G,\,\phi(z,v)=z\cdot v$ and the inverse map
$z\in G\mapsto z^{-1}\in G$ are para-holomorphic. 
\end{definition}

Let $U$ be a coordinate neighborhood of the identity $u_G$ of an $m$-parameter para-complex Lie group $G$. The coordinates of $u_G$ are identified with $\{0,\dots,0\}\in C^m$, while the coordinates of elements of $z$, $v$, $w$ of $U$ will be denoted by $\{z^a\}$, $\{v^a\}$, $\{w^a\}$, respectively, $a, b, c,\ldots\in\{1,\ldots,m\}$. The map $\phi:G\times G \rightarrow G$ given by $w=\phi(z,v)$ is represented para-holomorphically by $m$ equations $w^\alpha=\phi^a(z,v)$, in which $\{\phi^a\}$ denotes a set of $r$ para-complex-valued para-holomorphic functions on $G\times G$, where $\phi^a(z,v)$ is an abbreviated notation for $\phi^a(z^1,\dots,z^m,v^1,\dots,v^m)$. Since $z=u_G\cdot z=z\cdot u_G$ for all $z\in G$, it follows that up to and including second order terms
\begin{equation}\label{1.1}
w^a=\phi^a(z,v)=z^a+v^a+A^{a}_{bc}z^b v^c+\ldots,
\end{equation}
where the $3$--index symbols $A^{a}_{bc}$ are para-complex constants (in a given para-holomorphic coordinate system) in terms of which the structure constants of $G$ are defined as
$C^{a}_{bc}=A^{a}_{bc}-A^{a}_{cb}$.

Let us denote
\begin{equation}\label{1.4}
\Phi^a_b(z,v)=\frac{\partial\phi^a(z,v)}{\partial z^b},\ \ \ \Psi^a_b(z,v)=\frac{\partial\phi^a(z,v)}{\partial v^b},
\end{equation}
so that by \eqref{1.1}
\begin{equation}\label{1.5}
\Phi^a_b(z,0)=\Psi^a_b(0,v)=\delta^a_b.
\end{equation}
The derivatives (\ref{1.4}) give rise to the definitions of the following para-holomorphic functions on $G$:
\begin{equation}
\label{1.6}
\stackrel{*}{\chi^a_b}(z)=\Phi^a_b(0,z), \, \chi^a_b(z)=\Psi^a_b(z,0),\, \lambda^a_b(z)=\Phi^a_b(z,z^{-1}), \, \stackrel{*}{\lambda^a_b}(z)=\Psi^a_b(z,z^{-1}),
\end{equation}
it being noted as a direct consequence of (\ref{1.5})
\begin{equation}
\stackrel{*}{\chi^a_b}(0)=\chi^a_b(0)=\lambda^a_b(0)= \stackrel{*}{\lambda^a_b}(0)=\delta^a_b.
\end{equation}

Using the same technique as in the real (complex) case, \cite{Go, Ru, I-Io}, we obtain that $\stackrel{*}{\chi_{b}^{a}}(z)=\widetilde{\lambda}_{b}^{a}(z)$, where  $\widetilde{\lambda}_{b}^{a}(z)$ denotes the elements of the para-holomorphic matrix that is inverse to $(\lambda_{b}^{a}(z))$ and $\stackrel{*}{\lambda_{b}^{a}}(z)=\widetilde{\chi}_{b}^{a}(z^{-1})$, where  $\widetilde{\chi}_{b}^{a}(z)$ denotes the elements of the para-holomorphic matrix that is inverse to $(\chi_{b}^{a}(z))$. Also, we can consider the left and right invariant para-holomorphic $1$-forms on the para-complex Lie group $G$ defined by $\widetilde{\chi}^{a}=\widetilde{\chi}_{b}^{a}(z)dz^{b}$ and $\lambda^{a}=\lambda_{b}^{a}(z)dz^{b}$, respectively. Then
\begin{equation}
\label{3.32}
\Gamma^{a}_{bc}(z)=\widetilde{\lambda}^a_d(z)\left(\frac{\partial \lambda^d_b(z)}{\partial z^c}+\frac{1}{2}C^{d}_{pq}\lambda^p_b(z)\lambda^q_c(z)\right)=\frac{1}{2}\widetilde{\lambda}^a_d(z)\left(\frac{\partial\lambda^d_b(z)}{\partial z^c}+\frac{\partial\lambda^d_c(z)}{\partial z^b}\right),
\end{equation}
defines the local coefficients of an unique torsion-free para-holomorphic connection on $G$.

The torsion-free para-holomorphic connection from \eqref{3.32} is always metric with respect to the para-holomorphic tensor field $g\in \left(T^{1,0}G\right)^*\otimes \left(T^{1,0}G\right)^*$ whose local components are given by
\begin{equation}
\label{4.12}
g_{ab}(z)=C_{pq}\lambda_{a}^{p}(z)\lambda_{b}^{q}(z),
\end{equation}
where $C_{ab}=C^{c}_{ad}C^{d}_{bc}$ are the para-complex Cartan-Killing elements of the para-complex Lie group $G$.  Moreover, if the para-complex Lie group $G$ is semi-simple, that is $\det g_{ab}\neq0$ (or equivalently $\det C_{ab}\neq0$), it is the only symmetric para-holomorphic connection for which this is the case.

\begin{remark}
\label{r1}
For the case of para-holomorphic metric tensor $g$ from \eqref{4.12} its symmetry is guaranted from the expression of para-complex Cartan-Killing elements $C_{ab}$. If $G$ is semi-simple then the para-holomorphic connection coefficients from \eqref{3.32} admit a representation in terms of the para-complex Christoffel symbols of \eqref{4.12}.
\end{remark}
\begin{remark}
\label{r2}
The para-holomorphic metric tensor $g_{ab}$ from \eqref{4.12} is not in general unique such that the torsion-free para-holomorphic connection from \eqref{3.32} is metric with respect to it, (see the construction from the real case \cite{Ru}). 
\end{remark}

As usual, the para-holomorphic curvature tensor of the torsion-free para-holomorphic connection from \eqref{3.32} must be specified as
\begin{equation}
\label{5.11}
R^{d}_{c,ab}=-\frac{1}{4}\widetilde{\lambda}_{f}^{d}C^{f}_{pq}C^{q}_{rs}\lambda_{c}^{p}\lambda^{r}_{a}\lambda^{s}_{b}.
\end{equation}
Then, the para-holomorphic Ricci tensor is 
\begin{equation}
\label{5.12}
R_{ca}=R^{b}_{c,ab}=-\frac{1}{4}C^{f}_{pq}C^{q}_{rf}\lambda_{c}^{p}\lambda_{a}^{r},
\end{equation}
or, in terms of para-complex Cartan-Killing elements
\begin{displaymath}
R_{ca}=-\frac{1}{4}C_{pr}\lambda_{c}^{p}\lambda_{a}^{r}.
\end{displaymath}
By comaparing this para-holmorphic tensor with the para-holomorphic metric tensor from \eqref{4.12} it is seen that the para-holomorphic Ricci tensor satisfies 
\begin{equation}
\label{5.13}
R_{ab}=-\frac{1}{4}g_{ab},
\end{equation}
which implies that every para-complex Lie group is locally \textit{para-holomorphic Einsteinian}.

Now, as well as we noticed, if $G$ is a semi-simple para-complex Lie group the para-holomorphic metric tensor from \eqref{4.12} is symmetric and nondegenerated. Thus, according to discussion from the Section 4 
\begin{equation}
\label{ds}
ds^2=2 Re\left[g_{ab}(z)dz^{a}\otimes dz^{b}\right]
\end{equation}
defines a para-K\"{a}hlerian Norden metric on $G$. Consequently, we have
\begin{theorem}
\label{t3}
Every semi-simple para-complex Lie groups is  a para-K\"{a}hlerian Norden Einstein space with respect to the para-K\"{a}hlerian Norden metric defined by \eqref{ds}.
\end{theorem}
Also, the following proposition holds.
\begin{proposition}
If the para-complex Lie group is semi-simple then its para-holomorphic curvature scalar is constant and it is given by
\begin{equation}
\label{5.14}
g^{ab}R_{ab}=-\frac{1}{4}\left(\dim_CG\right).
\end{equation}
\end{proposition}

Moreover, it is natural to consider the type $(0,4)$ para-holomorphic curvature tensor associated with \eqref{4.12} and \eqref{5.11} as
\begin{equation}
\label{5.15}
R_{abcd}=g_{df}R^{f}_{c,ab},
\end{equation}
and, the explicit expression of this para-holomorphic tensor is given by
\begin{equation}
\label{5.16}
R_{abcd}=-\frac{1}{4}C_{tf}C^{f}_{pq}C^{q}_{rs}\lambda_{a}^{p}\lambda_{b}^{t}\lambda_{c}^{r}\lambda_{d}^{s}.
\end{equation}
The \textit{para-holomorphic sectional curvature} $k(Z,W)$ of $G$ with respect a pair of para-holomorphic vector fields $Z,W\in\Gamma_{ph}(T^{1,0}(G))$ can be written in accordance with the standard formula
\begin{equation}
\label{5.18}
k(Z,W)(g_{ac}g_{bd}-g_{ad}g_{bc})Z^{a}Z^{c}W^{b}W^{d}=R_{abcd}Z^{a}Z^{c}W^{b}W^{d}.
\end{equation}
Finally, we notice that similarly to the real case \cite{Ru}, the following two theorems hold.
\begin{theorem}
\label{t5}
The para-holomorphic sectional curvature of a para-complex Lie group $G$ with respect to every pair of right-invariant para-holomorphic vector fields is constant. 
\end{theorem}
\begin{theorem}
\label{t6}
The covariant derivatives of the components of $R^{d}_{c,ab}$ with respect to the torsion-free para-holomorphic connection from \eqref{3.32} vanish identically.
\end{theorem}

\noindent 
Cristian Ida, Alexandru Ionescu and Adelina Manea\\
Department of Mathematics and Computer Science\\
University Transilvania of Bra\c{s}ov\\
Address: Bra\c{s}ov 500091, Str. Iuliu Maniu 50,  Rom\^{a}nia\\
 email:\textit{cristian.ida@unitbv.ro; alexandru.codrin.ionescu@gmail.com; amanea28@yahoo.com}
\medskip

\end{document}